\documentclass[11pt]{amsart}
\usepackage[colorlinks,citecolor=red, dvipdfm]{hyperref}
\usepackage{extarrows}
\usepackage[all]{xy}

\newtheorem{theorem}{Theorem}[section]
\newtheorem{lemma}[theorem]{Lemma}

\theoremstyle{definition}

\newtheorem{question}[theorem]{Question}

\newtheorem{corollary}[theorem]{Corollary}
\newtheorem{remark}[theorem]{Remark}

\theoremstyle{remark}

\newcommand{\be}{\begin{equation}}
\newcommand{\ee}{\end{equation}}

\numberwithin{equation}{section}

\begin{document}
\title{On the spectral rigidity of Einstein-type K\"{a}hler manifolds}
\author{Ping Li}
\address{School of Mathematical Sciences, Tongji University, Shanghai 200092, China}

\email{pingli@tongji.edu.cn\\
pinglimath@gmail.com}

\thanks{The author was partially supported by the National
Natural Science Foundation of China (Grant No. 11722109).}




\begin{abstract}
We are concerned in this article with a classical question in spectral geometry dating back to McKean-Singer, Patodi and Tanno: whether or not the constancy of holomorphic sectional curvature of a complex $n$-dimensional compact K\"{a}hler manifold can be completely determined by the eigenvalues of its $p$-Laplacian for a \emph{single} integer $p$? We treat this question in this article under two Einstein-type conditions: cohomologically Einstein and Fano Einstein. Building on our previous work, we show that for cohomologically Einstein K\"{a}hler manifolds this is true for all but finitely many pairs $(p,n)$. As a consequence, the standard complex projective spaces can be characterized among cohomologically Einstein K\"{a}hler manifolds in terms of a single spectral set in all these cases. Moreover, in the case of $p=0$, we show that the complex projective spaces can be characterized among Fano K\"{a}hler-Einstein manifolds only in terms of the first nonzero eigenvalue with multiplicity, which has a similar flavor to a recent celebrated result due to Kento Fujita.
\end{abstract}

\maketitle


\section{Introduction and main results}\label{section1}
Let $(M,g)$ be an $m$-dimensional connected, closed and oriented Riemannian manifold, $\Omega^p(M)$ ($0\leq p\leq m$) the space of smooth exterior $p$-forms on $M$, $d:~\Omega^p(M)\rightarrow\Omega^{p+1}(M)$ the operator of exterior differentiation, and $d^{\ast}:~\Omega^p(M)\rightarrow\Omega^{p-1}(M)$ the formal adjoint of $d$ relative to the Riemannian metric $g$. Here $\Omega^p(M):=0$ provided that $p=-1$ or $m+1$. We have, For each $0\leq p\leq m$, the following second-order self-adjoint elliptic operator, the Laplacian acting on $p$-forms:
\be\label{delta}\Delta_p:=dd^{\ast}+d^{\ast}d:~\Omega^p(M)
\longrightarrow\Omega^p(M).\ee
It is well-known that $\Delta_p$ has an infinite discrete sequence
\be0\label{eigenvalue}\leq\lambda_{1,p}\leq\lambda_{2,p}\leq\cdots\leq\lambda_{k,p}\leq
\cdots\uparrow+\infty\ee
of eigenvalues and each of them is repeated as many times as its multiplicity indicates. These $\lambda_{k,p}$ are called the \emph{spectra} of $\Delta_p$. Put
$$\text{Spec}^p(M,g):=\big\{\lambda_{1,p}, \lambda_{2,p},\ldots,\lambda_{k,p},\ldots\big\},$$
which is called the spectral set of $\Delta_p$. Duality and Hodge theory tell us that $\text{Spec}^p(M,g)=\text{Spec}^{m-p}(M,g)$ and $0\in\text{Spec}^p(M,g)$ if and only if the $p$-th Betti number $b_p(M)\neq0$ and its multiplicity is precisely $b_p(M)$.

An important problem in spectral geometry is to investigate how the geometry of $(M,g)$ can be reflected by its spectra $\{\lambda_{k,p}\}$. In general the spectra $\{\lambda_{k,p}\}$ are not able to determine a manifold up to an isometry, as Milnor has constructed in \cite{Mi} two non-isometric Riemannian structures on a $16$-dimensional manifold such that for each $p$ the spectral sets $\text{Spec}^p(\cdot)$ with respect to these Riemannian metrics are the same. Nevertheless, we may still ask to what extent the spectra $\{\lambda_{k,p}\}$ encode the geometry of $(M,g)$.

Recall that, for any positive integer $N$, the famous Minakshisundaram-Pleijel asymptotic expansion formula, which is the integration on the asymptotic expansion of the heat kernel for Laplacian, tells us
\be\label{mpgformula}
\begin{split}
\text{Trace}(e^{-t\Delta_p})&=\sum_{k=0}^{\infty}\exp(-\lambda_{k,p}t)\\
&=\frac{1}{(4\pi t)^{\frac{m}{2}}}\Big[{m\choose p}\text{Vol}(M,g)+
\sum_{i=1}^{N}a_{i,p}t^i\Big]+O(t^{N-\frac{m}{2}+1}),\qquad t\downarrow0\\
&=\frac{1}{(4\pi t)^{\frac{m}{2}}}
\sum_{i=0}^{N}a_{i,p}t^i+O(t^{N-\frac{m}{2}+1}),\qquad t\downarrow0,\qquad\Big(a_{0,p}:={m\choose p}\text{Vol}(M,g)\Big).
\end{split}\ee
Here $\text{Vol}(M,g)$ is the volume of $(M,g)$ and $a_{i,p}$ ($i\geq 1$) are certain functions of the curvature, which are completely determined by the spectral set $\text{Spec}^p(M,g)$. The coefficients $a_{1,0}$ and $a_{2,0}$ were calculated by Berger and McKean-Singer (\cite{Be}, \cite{MS}) and then in \cite{Pa} Patodi explicitly determined $a_{1,p}$ and $a_{2,p}$ for \emph{all} $p$.

When $(M,g)$ is flat, i.e., has constant sectional curvature $c=0$, then $a_{i,p}=0$ for all $p$ and $i\geq 1$ as these $a_{i,p}$ are functions of the curvature. McKean and Singer raised in \cite{MS} a converse question: if $a_{i,0}=0$ for all $i\geq 1$, then whether or not $(M,g)$ is flat? They proved in \cite{MS} that this is true if the dimension $m\leq 3$. Patodi further showed in \cite{Pa} that this is true if $m\leq 5$ and is false when $m>5$ by constructing counterexamples (\cite[p. 283]{Pa} or \cite[p. 65]{Pa2}). This means that in general the vanishing of $a_{i,p}$ ($i\geq 1$) for only \emph{one single} value $p=0$ is not enough to derive the flatness. Nevertheless, applying the explicit expressions of $a_{1,p}$ and $a_{2,p}$ determined by himself in \cite{Pa}, Patodi showed that whether or not $(M,g)$ is of constant sectional curvature $c$ is completely determined by the quantities $\{a_{i,p}~|~i=0,1,2,~p=0,1\}$, i.e., by the spectral sets $\text{Spec}^0(M,g)$ and $\text{Spec}^1(M,g)$ (\cite[p. 281]{Pa} or \cite[p. 63]{Pa2}).

The notion of ``holomorphic sectional curvature" (``HSC" for short) in K\"{a}hler geometry is the counterpart of that of ``sectional curvature" in Riemannian geometry and
so it is natural to consider a similar question on K\"{a}hler manifolds. Note that if two Riemannian manifolds have the same spectral set $\text{Spec}^p(\cdot)$ for \emph{some} $p$, then due to the asymptotic formula $(\ref{mpgformula})$ they necessarily have the same dimension. Note also that for an $m$-dimensional Riemannian manifold we only need to consider the spectral sets $\text{Spec}^p(\cdot)$ for $p\leq[\frac{m}{2}]$ as $\text{Spec}^p(\cdot)=\text{Spec}^{m-p}(\cdot)$. In view of these two basic facts, we can now pose the following question in the K\"{a}hler version, which was initiated by Tanno in \cite{Ta1}.
\begin{question}\label{Q1}
Suppose that $(M_1,g_1,J_1)$ and $(M_2,g_2,J_2)$ are two complex $n$-dimensional compact K\"{a}hler manifolds such that $\text{Spec}^p(M_1,g_1)=\text{Spec}^p(M_2,g_2)$
for a \emph{fixed} $p$ with $p\leq n$. Is it true that $(M_1,g_1,J_1)$ is of constant HSC $c$ if and only if $(M_2,g_2,J_2)$ is so?
\end{question}
 Recall that, up to a holomorphic isometry, $(\mathbb{C}P^n(c),g_0,J_0)$, the standard complex $n$-dimensional projective space equipped with the Fubini-Study metric with \emph{positive} constant HSC $c$, is the unique complex $n$-dimensional compact K\"{a}hler manifold with positive constant HSC $c$ by the classical uniformization theorem. So we also have the following spectral characterization problem for $(\mathbb{C}P^n(c),g_0,J_0)$, which was first explicitly proposed by B.Y. Chen and Vanhecke in \cite{CV}.
\begin{question}\label{Q2}
Suppose that $(M,g,J)$ is a compact K\"{a}hler manifolds such that $\text{Spec}^p(M,g)=\text{Spec}^p(\mathbb{C}P^n(c),g_0)$
for a \emph{fixed} $p$ with $p\leq n$. Is it true that $(M,g,J)$ is holomorphically isometric to $(\mathbb{C}P^n(c),g_0,J_0)$?
\end{question}
Clearly a positive answer to Question \ref{Q1} implies that to Question \ref{Q2}.
Tanno showed in \cite[p. 402]{Ta1} that Question \ref{Q1} is true for $(p=0, n\leq 5)$ and $(p=0, n\leq 6)$ provided that the constant HSC $c\neq 0$. In \cite[p. 129]{Ta2} he further showed that Question \ref{Q1} is true for $(p=1, 8\leq n\leq 51)$. Consequently, Question \ref{Q2} is also true in these cases (\cite[Theorem D]{Ta1}, \cite[p. 129]{Ta2}). Chen and Vanhecke showed in \cite{CV} that Question \ref{Q2} is true for $(p=2, \text{all $n$ except $n=8$})$.
Besides these results, Question \ref{Q2} was also treated in some other literature and various results were claimed (see Remark \ref{remarks}) but unfortunately their proofs contain various mistakes and/or gaps, which have been clarified recently in \cite[\S 2.3]{Li3}. The purpose of the work in \cite{Li3} is two-folds: to clarify some gaps in previously existing literature related to Question \ref{Q2}, and to settle Question \ref{Q2} down affirmatively for each \emph{positive} and \emph{even} $p$ in all dimensions $n$ with at most two exceptions (\cite[Theorem 1.3]{Li3}).

As mentioned above, without any extra condition, a single spectral set is in general \emph{not} enough to derive the constancy of sectional curvature in the Riemannian case. Nevertheless, Sakai showed that, with the condition of $(M,g)$ being \emph{Einstein}, $\text{Spec}^0(M,g)$ is indeed enough to derive the desired conclusion (\cite[Theorem 5.1]{Sa}). \emph{The purpose of the present work} is to treat this similar question for K\"{a}hler manifolds. Recall that a compact K\"{a}hler manifold is called \emph{cohomologically Einstein} if its first Chern class and K\"{a}hler class are proportional. With this notion understood, building on the work in \cite{Li3}, \emph{our first main result} in this article is the following
\begin{theorem}\label{theorem1}
Suppose that $(M_1,g_1,J_1)$ and $(M_2,g_2,J_2)$ are two complex $n$-dimensional compact K\"{a}hler manifolds such that $$a_{i,p}(M_1,g_1)=a_{i,p}(M_2,g_2),\qquad i=0,1,2,$$
for a fixed $p$ with $p\leq n$, $(M_1,g_1,J_1)$ is cohomologically Einstein, and $(M_2,g_2,J_2)$ is of constant HSC $c$. Then $(M_1,g_1,J_1)$ is of constant HSC $c$ if the pair $(p,n)$ satisfies one of the following conditions:
\begin{enumerate}
\item
$p=0$ and $n\geq 1$;

\item
$p=1$ and $n\geq 6$;

\item
$p=2$ and $n\neq8$;

\item
$p\geq 3$ and $p^2-2np+\frac{n(2n-1)}{3}\neq0$.
\end{enumerate}
\end{theorem}

\begin{remark}
Here the reason that the exceptional cases $(p=1, n<6)$ be not able to be dealt with is due to the negativity of some quantity related to $p$ and $n$ in these cases, which is required to be positive in our proof.
The requirement that
$$p^2-2np+\frac{n(2n-1)}{3}\neq0$$
arises from a constant in front of the expression $a_{1,p}$. \big(see (\ref{integralformula2.5})\big)
\end{remark}
It turns out in \cite[\S 5.2]{Li3} that the positive integer solutions $(p,n)$ to the equation
$$p^2-2np+\frac{n(2n-1)}{3}=0$$
with $p\leq n$ are precisely parametrized by positive integers $k$, denoted by $(p_k,n_k)$, and satisfy the following recursive formula
\begin{eqnarray}\label{recursiveformula}
\left\{ \begin{array}{ll}
({p}_1,{n}_1)=(1,3),\\
~\\
{p}_{k+1}=8{n}_k-5{p}_k+1,\\
~\\
{n}_{k+1}=19{n}_k-12{p}_k+3.
\end{array} \right.
\end{eqnarray}
Direct calculations show that
$$(p_2,n_2)=(20,48), (p_3,n_3)=(285,675), (p_4,n_4)=(3976,9408), (p_5, n_5)=(55385,131043), \cdots,$$ whose distributions become more and more sparse as $k\rightarrow\infty$.

Theorem \ref{theorem1} and the recursive formula (\ref{recursiveformula}) imply the following result
\begin{corollary}\label{coro1}
Suppose that $(M_1,g_1,J_1)$ and $(M_2,g_2,J_2)$ are two complex $n$-dimensional compact K\"{a}hler manifolds such that $\text{Spec}^p(M_1,g_1)=\text{Spec}^p(M_2,g_2)$
for a fixed $p$ with $p\leq n$, $(M_1,g_1,J_1)$ is cohomologically Einstein, and $(M_2,g_2,J_2)$ is of constant HSC $c$. Then $(M_1,g_1,J_1)$ is of constant HSC $c$ if the pair $(p,n)$ satisfies one of the following cases:
\begin{enumerate}
\item
($p=0$, all dimensions $n$);

\item
($p=1$, all dimensions $n\geq 6$);

\item
($p=2$,
all dimensions $n$ with at most one exception $n=8$);

\item
($p\geq 3$ and $p\not\in\{p_k~|~k\geq 2\}$, all dimensions $n$);

\item
($p=p_k$, all dimensions $n$ with at most one exception $n=n_k$)~\rm($k\geq 2$\rm).
\end{enumerate}
Here $(p_k,n_k)$ \rm($k\geq 2$\rm) are determined by (\ref{recursiveformula}).
\end{corollary}
Consequently, Corollary \ref{coro1} can be carried over to yield the same result when the HSC $c>0$, which amounts to $(M_2,g_2,J_2)=(\mathbb{C}P^n(c),g_0,J_0)$. In this situation we can, however, do one more case. Note that the exceptional case $(p=2,n=8)$ is not able to be dealt with in Corollary \ref{coro1} due to the vanishing of a coefficient in the proof, which would be clear later (Lemma \ref{lastlemma}). Nevertheless, thanks to a recent breakthrough due to Fujita (\cite{Fu}) solving a long-standing conjecture in complex geometry, the difficulty in this exceptional case for $(\mathbb{C}P^n(c),g_0,J_0)$ can be successfully overcome, which has been explained in \cite{Li3} and shall be briefly reviewed again at the end of Section \ref{section3}, after the proof of Theorem \ref{theorem1} as well as Corollary \ref{coro1}. In summary, we have the following partial affirmative answer towards Question \ref{Q2}.
\begin{corollary}\label{coro2}~
Suppose that $(M,g,J)$ is a complex $n$-dimensional compact cohomologically Einstein K\"{a}hler manifolds such that $\text{Spec}^p(M,g)=\text{Spec}^p(\mathbb{C}P^n(c),g_0)$
for a fixed $p$ with $p\leq n$. Then $(M,g,J)$ is holomorphically isometric to $(\mathbb{C}P^n(c),g_0,J_0)$ if the pair $(p,n)$ satisfies one of the following cases:
\begin{enumerate}
\item
($p=0$, all dimensions $n$);

\item
($p=1$, all dimensions $n\geq 6$);

\item
($p=2$,
all dimensions $n$);

\item
($p\geq 3$ and $p\not\in\{p_k~|~k\geq 2\}$, all dimensions $n$);

\item
($p=p_k$, all dimensions $n$ with at most one exception $n=n_k$)~\rm($k\geq 2$\rm).
\end{enumerate}
Here $(p_k,n_k)$ \rm($k\geq 2$\rm) are determined by (\ref{recursiveformula}).
\end{corollary}
When $p$ is \emph{even} and \emph{positive} and $\text{Spec}^p(M,g)=\text{Spec}^p(\mathbb{C}P^n(c),g_0)$, , it turns out in \cite[Lemma 4.3]{Li3} that the condition  of $(M,g,J)$ being cohomologically Einstein is automatically satisfied due to the Hard Lefschetz theorem.
Also note that the positive integers $p_k$ determined by the recursive formula (\ref{recursiveformula}) are even if and only if $k$ are even. Hence we have the following affirmative answer to Question \ref{Q2} for the following $(p,n)$ \emph{without any extra condition}, which is precisely the main result in \cite{Li3}.
\begin{corollary}\label{cc}
Assume that $p$ is even, positive and $p\leq n$. Then
\begin{enumerate}
\item
for $p=2$, Question \ref{Q2}
holds in all dimensions $n$;

\item
for $p\geq 4$ and $p\not\in\{p_{2k}~|~k\geq 1\}$, Question \ref{Q2} holds in all dimensions $n$;

\item
for $p=p_{2k}$, Question \ref{Q2} holds in all dimensions $n$ with at most one exception $n=n_{2k}$.~\rm($k\geq 1$\rm).
\end{enumerate}
Here $(p_{2k},n_{2k})$ \rm($k\geq 1$\rm) are determined by (\ref{recursiveformula}).
\end{corollary}

\begin{remark}\label{remarks}
As previously mentioned, Chen and Vanhecke settled Question \ref{Q2} in \cite{CV} for the cases ($p=2,$ all $n$ except $n=8$) in Corollary \ref{cc}. The exceptional case $(p=2,n=8)$ left in \cite{CV} was treated by Goldberg in \cite{Go}.
The main result in Corollary \ref{coro2} was also claimed by Gauchman and Goldberg in \cite[Theorem 1]{GG}. Unfortunately the proofs in \cite{Go} and \cite{GG} contain several gaps, which have been clarified in \cite{Li3} (cf. \cite[\S 2.3, Remark 4.2]{Li3}). Nevertheless, the proofs in \cite{Go} and \cite{GG} still contain invaluable ideas, which, together with the recent result of Fujita in \cite{Fu}, inspired our work \cite{Li3}.
\end{remark}

We now state our second main result closely related to Question \ref{Q2} in this article as well as the main result in \cite{Fu}. The case $p=0$ is particularly interesting as the spectral set $\text{Spec}^0(\cdot)$ consists of the eigenvalues of the Laplacian on functions, and so it is more important to see if Questions \ref{Q1} and \ref{Q2} are true when $p=0$. As previously noted, we only know from \cite{Ta1} that they hold in low dimensions. As is well-known among the set $\text{Spec}^0(\cdot)$ the first nonzero eigenvalue, which is $\lambda_{2,0}$ in our notation of (\ref{eigenvalue}), plays fundamental roles in various aspects in differential geometry. With this fact in mind, applying an integral formula of Bochner type essentially due to Lichnerowicz (\cite{Lic}) as well as a result of Tanno in \cite{Ta}, we shall show the following \emph{second main result} in this article.
\begin{theorem}\label{mul}
Assume that $(M,g,J)$ is a complex $n$-dimensional Fano K\"{a}hler-Einstein manifold such that its scalar curvature is normalized to be that of $(\mathbb{C}P^n(c),g_0,J_0)$, i.e., $s_g=s_{g_0}=n(n+1)c$. If the first nonzero eigenvalue $\lambda_{2,0}(M,g)$ and its multiplicity of $(M,g,J)$ are the same as those of $(\mathbb{C}P^n(c),g_0,J_0)$, then $(M,g,J)$ is holomorphically isometric to $(\mathbb{C}P^n(c),g_0,J_0)$.
\end{theorem}
\begin{remark}
As previously mentioned, our treatment of the exceptional case $(p,n)=(2,8)$ in Corollaries \ref{coro2} and \ref{cc} relies on Fujita's recent result \cite[Theorem 1.1]{Fu}, which characterizes the standard complex projective spaces among Fano K\"{a}ler-Einstein manifolds in terms of their volumes when their metrics are normalized with the same constant scalar curvatures (cf. \cite[Theorem 2.2]{Li3}). Therefore Theorem \ref{mul} can also be compared to this result as another characterization in terms of the first nonzero eigenvalues with multiplicity.
\end{remark}
The rest of this article is structured as follows. We recall in Section \ref{section2} some necessary notation and integral formulas set up in \cite{Li3} and prove the main result, Theorem \ref{theorem1}, in Section \ref{section3}.
Section \ref{section4} is then devoted to the proof of Theorem \ref{mul}.
\section{Preliminaries}\label{section2}
In this section we shall recall some necessary notation and  integral formulas involving in the curvature on compact K\"{a}hler manifolds, which rely on the tools developed in \cite{Li3}, and we refer the reader to \cite{Li3} for more related details.

Assume now that $(M,g,J)$ is a compact K\"{a}hler manifold with complex dimension $n\geq2$, i.e., $J$ is an integrable complex structure and $g$ a $J$-invariant Riemannian metric. Define
\begin{eqnarray}\label{kahlerricciformtensor}
\left\{ \begin{array}{ll}
\text{$\omega:=\frac{1}{2\pi}g(J\cdot,\cdot)$, the K\"{a}hler form of $g$,}\\
~\\
\text{Ric$(g)$:=the Ricci tensor of $g$,}\\
~\\
\text{Ric$(\omega):=\frac{1}{2\pi}
$Ric$(g)(J\cdot,\cdot)$, the Ricci form of $g$,}\\
~\\
\text{$s_g:={\rm Trace}_g{\rm Ric}(g)$, the scalar curvature of $g$,}\\
~\\
\widetilde{\text{R}}\text{ic}(\omega):=\text{Ric}(\omega)-\frac{s_g}{2n}\omega,~\text{the traceless part of Ric$(\omega)$.}
\end{array} \right.
\end{eqnarray}

It is well-known that the Ricci form $\text{Ric$(\omega)$}$ is a closed form representing the first Chern class $c_1(M)$, $g$ is Einstein if and only if $\widetilde{\text{R}}\text{ic}(\omega)\equiv0$, and, in our notation of $\omega$,
\be\label{volumeelement}\text{the volume element of $(M,g):={\rm dvol}=\frac{\pi^n}{n!}\omega^n$}.\ee
Recall that the K\"{a}hler curvature tensor of $g$, which is the complexification of its Riemannian curvature tensor and denoted by $R^{c}$, splits into three irreducible components under the unitary group action: $R^{c}=S^{c}+P^{c}+B,$
where $S^{c}$, $P^{c}$ and $B$ involve respectively the scalar curvature part, the traceless Ricci tensor part and the Bochner curvature tensor.
The K\"{a}hler metric $g$ is of constant HSC if and only if it is Einstein and has the vanishing Bochner curvature tensor, i.e., if and only if $\widetilde{\text{R}}\text{ic}(\omega)\equiv0$ and $B\equiv0$. For more details on these tensors and their relations with those in the Riemannian setting, we refer the reader to \cite[\S3.1]{Li3}.

With the notation understood, we have the following integral formulas (cf. \cite[Lemma 3.5, Lemma 4.1]{Li3}).
\begin{lemma}\label{integralformulas}
Suppose that $(M,g,J)$ is a compact K\"{a}hler manifold with complex dimension $n\geq2$. Then
\be\label{integralformula1}
\int_Mc_1(M)\wedge[\omega]^{n-1}=
\int_M\text{Ric}(\omega)\wedge\omega^{n-1}=
\frac{1}{2n}\int_M{s_g}\cdot\omega^n,
\ee

\be\label{integralformula2}
\int_Mc_1^2(M)\wedge[\omega]^{n-2}=\int_M\text{Ric}(\omega)^2\wedge\omega^{n-2}=
\int_M
\big(\frac{n-1}{4n}s^2_g-|{\rm\tilde{R}ic}(\omega)|^2\big)
\cdot\frac{\omega^n}{n(n-1)},\ee

\begin{eqnarray}\label{integralformula2.5}
\left\{ \begin{array}{ll}
a_{0,p}(M,g)={2n\choose p}{\rm Vol}(M,g)\\
~\\
a_{1,p}(M,g)=\frac{(2n-2)!}{p!(2n-p)!}\big[p^2-2np+
\frac{n(2n-1)}{3}\big]\int_Ms_g{\rm dvol},
\end{array} \right.
\end{eqnarray}
and
\be\label{integralformula3}
\begin{split}
&a_{2,p}(M,g)\\
=&\int_M
\Big[\big(\frac{2}{n(n+1)}\lambda_1+
\frac{1}{2n}\lambda_2+\lambda_3\big)s_{g}^2+
\big(\frac{16}{n+2}\lambda_1
+2\lambda_2\big)|{\rm\widetilde{R}ic}(\omega)|^2+
4\lambda_1|B|^2\Big]{\rm dvol},\end{split}
\ee
where $|{\rm\tilde{R}ic}(\omega)|^2$ and $|B|^2$ are their pointwise squared norms and
\begin{eqnarray}\label{patodicoefficient}
\left\{ \begin{array}{ll} \lambda_1=\frac{1}{180}{2n\choose p}
-\frac{1}{12}{2n-2\choose p-1}+\frac{1}{2}{2n-4\choose p-2},\\
~\\
\lambda_2=-\frac{1}{180}{2n\choose p}
+\frac{1}{2}{2n-2\choose p-1}-2{2n-4\choose p-2},\\
~\\
\lambda_3=\frac{1}{72}{2n\choose p}
-\frac{1}{6}{2n-2\choose p-1}+\frac{1}{2}{2n-4\choose p-2}.
\end{array}\right.
\end{eqnarray}
In particular, if $g$ is of constant HSC $c$, then $s_g=n(n+1)c$ and thus (\ref{integralformula3}) becomes
\be\label{integralformula4}
a_{2,p}(M,g)=\int_M
\big(\frac{2}{n(n+1)}\lambda_1+
\frac{1}{2n}\lambda_2+\lambda_3\big)[n(n+1)c]^2{\rm dvol},~(\text{$g$: constant HSC $c$}).
\ee
\end{lemma}
\begin{remark}~
\begin{enumerate}
\item
(\ref{integralformula1}) and (\ref{integralformula2}) are essentially due to Apte in \cite{Ap}. For more details and remarks on (\ref{integralformula1}) and (\ref{integralformula2}), we refer the reader to \cite[Remark 3.6]{Li3}.

\item
The explicit formulas for $a_{2,p}$ as well as $a_{1,p}$ was calculated by Patodi (\cite[p. 277]{Pa} or \cite[p. 59]{Pa2}) in terms of various norms in the Riemannian setting. The relations between various norms arising from the curvature in Riemannian and K\"{a}hler manifolds were carefully investigated in \cite[\S3.1]{Li3} and the current formula (\ref{integralformula3}) in the K\"{a}hler version was obtained in \cite[Lemma 4.1]{Li3}.

\item
The factorial $t!$ and binomial symbol ${u\choose v}$ in (\ref{integralformula2.5}) and (\ref{patodicoefficient}) are understood to be $1$, $1$ and $0$ if respectively $t=0$, $v=0$ and $v<0$.
\end{enumerate}
\end{remark}

\section{Proof of Theorem \ref{theorem1}}\label{section3}
With the preliminaries in Section \ref{section2} in hand, we can now proceed to prove Theorem \ref{theorem1}.

We always assume in the sequel that the two complex $n$-dimensional compact K\"{a}hler manifolds ($n\geq2$) $(M_1,g_1,J_1)$ and $(M_2,g_2,J_2)$ satisfy the conditions assumed in Theorem \ref{theorem1}. Namely, $(M_1,g_1,J_1)$ is cohomologically Einstein, $(M_2,g_2,J_2)$ is of constant HSC $c$ and $a_{i,p}(M_1,g_1)=a_{i,p}(M_2,g_2)$ for $i=0,1,2$. Denote by the symbols $s_{g_i}$, $\omega_i$, $B_i$, etc. the corresponding quantities on $(M_i,g_i,J_i)$ ($i=1,2$).

The first observations are the following facts deriving from $a_{i,p}(M_1,g_1)=a_{i,p}(M_2,g_2)$ for $i=0,1$.
\begin{lemma}\label{keylemma}
Assume that $p^2-2np+\frac{n(2n-1)}{3}\neq0$. Then
\begin{enumerate}
\item
\be\label{key1}\int_{M_1}\big\{s_{g_1}^2-[n(n+1)c]^2\big\}{\rm dvol}\geq0,\ee
with equality if and only if the scalar curvature $s_{g_1}=n(n+1)c$ is a constant.

\item
\be\label{key2}\int_{M_1}|{\rm\tilde{R}ic}(\omega_1)|^2=
\frac{n-1}{4n}\int_{M_1}\big\{s_{g_1}^2-[n(n+1)c]^2\big\}{\rm dvol}.\ee
\end{enumerate}
\end{lemma}
\begin{proof}
Under the assumptions and via (\ref{integralformula2.5}), we have
\be\label{a01}\text{Vol}(M_1,g_1)=\text{Vol}(M_2,g_2),
\qquad\int_{M_1}s_{g_1}{\rm dvol}=\int_{M_2}n(n+1)c{\rm dvol}.\ee
Therefore,
\be\begin{split}
\int_{M_1}s_{g_1}^2\text{dvol}\geq\frac{(\int_{M_1}s_{g_1}
\text{dvol})^2}{\text{Vol}(M_1,g_1)}&=
\frac{(\int_{M_2}n(n+1)c
\text{dvol})^2}{\text{Vol}(M_2,g_2)}\qquad\big((\ref{a01})\big)\\
&=\int_{M_2}[n(n+1)c]^2\text{dvol}\\
&=\int_{M_1}[n(n+1)c]^2\text{dvol},\qquad\big((\ref{a01})\big)
\end{split}\nonumber\ee
where the equality holds if and only if $s_{g_1}$ is a constant and hence $s_{g_1}=n(n+1)c$. This completes the first part in this lemma. For the second part, note in (\ref{volumeelement}) that $\omega_i^n$ ($i=1,2$) are volume forms up to a universal constant and $(M_1,g_1,J_1)$ being cohomologically Einstein means that $c_1(M_1)\in\mathbb{R}[\omega_1]$. Therefore
\be\label{ric}
\begin{split}
\int_{M_1}
\big(\frac{n-1}{4n}s^2_{g_1}-|{\rm\widetilde{R}ic}(\omega_1)|^2\big)
\cdot\frac{\omega_1^n}{n(n-1)}
=&
\big(\int_{M_1}c_1^2(M_1)\wedge[\omega_1]^{n-2}\big)\qquad\big((\ref{integralformula2})\big)\\
=&\frac{\big(\int_{M_1}c_1(M_1)\wedge[\omega_1]^{n-1}\big)^2}
{\int_{M_1}\omega_1^n}\qquad\big(c_1(M_1)\in\mathbb{R}[\omega_1]\big)\\
=&\frac{\big(\int_{M_1}s_{g_1}\omega_1^{n}\big)^2}
{4n^2\int_{M_1}\omega_1^n}\qquad\big((\ref{integralformula1})\big)\\
=&\frac{\big(\int_{M_2}n(n+1)c\omega_2^{n}\big)^2}
{4n^2\int_{M_2}\omega_2^n}\qquad\big((\ref{a01})\big)\\
=&\frac{[n(n+1)c]^2}{4n^2}\int_{M_2}\omega_2^n\\
=&\frac{[n(n+1)c]^2}{4n^2}\int_{M_1}\omega_1^n.\qquad\big((\ref{a01})\big)
\end{split}
\ee
Now rewriting (\ref{ric}) by singling out the term $|{\rm\widetilde{R}ic}(\omega_1)|^2$ yields the desired equality (\ref{key2}).
\end{proof}

Together with (\ref{key2}) in Lemma \ref{keylemma}, the assumed condition $a_{2,p}(M_1,g_1)=a_{2,p}(M_2,g_2)$ yields the following key equality.
\begin{lemma}\label{a2a}
Assume that $p^2-2np+\frac{n(2n-1)}{3}\neq0$. Then
\be\label{a2b}
\big[\frac{4n+2}{(n+1)(n+2)}\lambda_1+
\frac{1}{2}\lambda_2+\lambda_3\big]\int_{M_1}\big\{s_{g_1}^2-[n(n+1)c]^2\big\}
{\rm dvol}+4\lambda_1\int_{M_1}|B_1|^2{\rm dvol}=0.
\ee
\end{lemma}
\begin{proof}
The condition $a_{2,p}(M_1,g_1)=a_{2,p}(M_2,g_2)$ and the expressions (\ref{integralformula3}) and (\ref{integralformula4}) for $M_1$ and $M_2$ tell us that
\be\label{key3}\begin{split}
0=&\int_{M_1}\big(\frac{2}{n(n+1)}\lambda_1+
\frac{1}{2n}\lambda_2+\lambda_3\big)\big\{s_{g_1}^2-[n(n+1)c]^2\}{\rm dvol}\\
&+\big(\frac{16}{n+2}\lambda_1
+2\lambda_2\big)\int_{M_1}|{\rm\tilde{R}ic}(\omega_1)|^2{\rm dvol}+
4\lambda_1\int_{M_1}|B_1|^2{\rm dvol}\\
=&\int_{M_1}\big(\frac{2}{n(n+1)}\lambda_1+
\frac{1}{2n}\lambda_2+\lambda_3\big)\big\{s_{g_1}^2-[n(n+1)c]^2\}{\rm dvol}\\
&+\big(\frac{16}{n+2}\lambda_1
+2\lambda_2\big)\Big\{\frac{n-1}{4n}\int_{M_1}\big\{s_{g_1}^2-[n(n+1)c]^2\big\}{\rm dvol}\Big\}+
4\lambda_1\int_{M_1}|B_1|^2{\rm dvol}\qquad\big((\ref{key2})\big)\\
=&\big(\frac{4n+2}{(n+1)(n+2)}\lambda_1+
\frac{1}{2}\lambda_2+\lambda_3\big)\int_{M_1}\big\{s_{g_1}^2-[n(n+1)c]^2\big\}
{\rm dvol}+4\lambda_1\int_{M_1}|B_1|^2{\rm dvol}.
\end{split}\ee
This yields the desired equality (\ref{a2b}).
\end{proof}
The inequality (\ref{key1}) and equality \ref{a2b} allow us to affirmatively solve Question \ref{Q1} in the following situations.
\begin{lemma}\label{lemmaend}
If the pair $(p,n)$ satisfies $p^2-2np+\frac{n(2n-1)}{3}\neq0$ and
\be\label{numerical}
\frac{4n+2}{(n+1)(n+2)}\lambda_1+
\frac{1}{2}\lambda_2+\lambda_3>0,\qquad\lambda_1>0,\ee
then $(M_1,g_1,J_1)$ is of constant HSC $c$.
\end{lemma}
\begin{proof}
Put (\ref{key1}), (\ref{a2b}) and Lemma \ref{lemmaend} together, we deduce that
\be\label{vanish1}\int_{M_1}\big\{s_{g_1}^2-[n(n+1)c]^2\big\}
{\rm dvol}=0, \qquad B_1\equiv0.\ee
The two equalities in (\ref{vanish1}), together with the equality characterization in (\ref{key1}), imply that the scalar curvature $s_{g_1}=n(n+1)c$ is constant and the Bochner curvature tensor $B_1$ vanishes. However, it is well-known that the constancy of $s_{g_1}$ and $c_1(M_1)\in\mathbb{R}[\omega_1]$ imply that $g_1$ is necessarily Einstein (cf. \cite[p. 19]{Ti}). Therefore the K\"{a}her metric $g_1$ is Einstein and has vanishing Bochner curvature tensor and hence of constant HSC, whose value is exactly $c$ as $s_{g_1}=n(n+1)c$.
\end{proof}
At last, we arrive at the proof of Theorem \ref{theorem1} by showing the following technical result.
\begin{lemma}\label{lastlemma}
\be\begin{split}
&\Big\{(p,n)~\big|~\text{$0\leq p\leq n,$, $n\geq 2$, and satify $(\ref{numerical})$}\Big\}\\
=&\Big\{(p=0,~n\geq2),~
(p=1,~n\geq6),~(p=2,~\text{$n\geq 2$ and $n\neq8$}),~(\text{$p\geq3$,~all $n\geq p$})\Big\}.
\end{split}\nonumber\ee
\end{lemma}
\begin{proof}
For $p=0$ and $p=1$, we can easily check that exactly those $n$ with $n\geq 2$ and $n\geq 6$ respectively satisfy these restrictions. For $n\geq 2$ and $p\in[2,2n-2]$, we showed in detail in \cite[Prop. 4.5, \S 5.1]{Li3} that
\be\label{pgeq2}\frac{4n+2}{(n+1)(n+2)}\lambda_1+
\frac{1}{2}\lambda_2+\lambda_3>0,\qquad\qquad\lambda_1\geq0,\ee
with $\lambda_1=0$ if and only if $(p,n)=(2,8)$.
\end{proof}
\begin{remark}
Although we assume the evenness of $p$ in \cite[Prop. 4.5]{Li3} to be compatible with the statement in \cite[Theorem 1.2]{Li3}, we can see through the proof in \cite[\S 5.1]{Li3} that it plays no role and (\ref{pgeq2}) even holds for any \emph{real number} $p\in[2,2n-2]$.
\end{remark}

Now via Lemmas \ref{lemmaend} and \ref{lastlemma} the proof of Theorem \ref{theorem1} is completed and consequently so is Corollary \ref{coro1}.

Let us end our proof of Corollary \ref{coro2} by briefly indicating that how the exceptional case $(p=2,n=8)$ can be dealt with in the case of $c>0$, i.e., in the case of $(M_2,g_2,J_2)=(\mathbb{C}P^8(c),g_0,J_0)$, due to a recent result of Fujita (\cite{Fu}), which has been explained in detail in \cite{Li3}. If $(p,n)=(2,8)$, then $\lambda_1=0$, i.e., in (\ref{a2b}) the coefficient in front of the term $\int_{M_1}|B_1|^2{\rm dvol}$ vanishes and from the proof of Lemma \ref{lemmaend} we can \emph{not} conclude that $B_1\equiv0$ but \emph{only} conclude that the K\"{a}hler metric $g_1$ is Einstein with $s_{g_1}=n(n+1)c$. Nevertheless, if the constant HSC $c$ in question is \emph{positive}, then in this case $(M_1,g_1,J_1)$ is a \emph{Fano} K\"{a}hler-Einstein manifold. Then an equivalent form of the main result in \cite{Fu} (cf. \cite[Theorem 2.2]{Li3} and the remarks before it) tells us that $(M_1,g_1,J_1)$ is holomorphically isometric to $(\mathbb{C}P^8(c),g_0,J_0)$.

\section{Proof of Theorem \ref{mul}}\label{section4}
\subsection{Preliminaries on vector fields and $1$-forms}

Assume throughout this subsection that $(M,g,J)$ is a complex $n$-dimensional compact K\"{a}hler manifold. In order to show Theorem \ref{mul}, we need to recall some classical facts and results related to complex-valued vector fields and $1$-forms on compact K\"{a}hler manifolds.

Due to the K\"{a}hlerness, we can choose a (locally defined) orthonormal frame field of the Riemannian manifold $(M,g)$ in such a manner: $\{e_i,e_{i+n}=Je_{i}~|~1\leq i\leq n\}$. Then
$$\big\{u_i:=\frac{1}{\sqrt{2}}(e_i-\sqrt{-1}Je_i)~|~1\leq i\leq n\big\}$$
is a $(1,0)$-type unitary frame field. Denote by $\{\theta^i~|~1\leq i\leq n\}$ the $(1,0)$-type unitary coframe field dual to $\{u_i\}$.

There is a one-to-one correspondence between complex-valued vector fields $X$ and $1$-forms $\xi$ via the K\"{a}hler metric $g$ by $\xi(Y)=g(X,Y)$ for all complex vector fields $Y$. We denote by ``$X\longleftrightarrow\xi$" this correspondence. If we decompose $X$ and $\xi$ into $(1,0)$ and $(0,1)$-types: $X=X^{(1,0)}+X^{(0,1)}$ and $\xi=\xi^{(1,0)}+\xi^{(0,1)}$, then $X^{(1,0)}\longleftrightarrow\xi^{(0,1)}$ and $X^{(0,1)}\longleftrightarrow\xi^{(1,0)}$.
To be more explicit, $$X^{(1,0)}=\sum_{i=1}^n\alpha_iu_i\longleftrightarrow\xi^{(0,1)}=\sum_{i=1}^n\alpha_i\overline{\theta^i},~\alpha_i\in\mathbb{C}.$$
A vector field $X$ is called \emph{real holomorphic} if it is real-valued and its $(1,0)$-part $X^{(1,0)}$ is holomorphic in the usual sense.

Let $\nabla$ be the complexified Levi-Civita connection on $(M,g,J)$ and write $\nabla=\nabla^{\prime}+\nabla^{\prime\prime}$, where
$\nabla^{\prime}=\sum_{i=1}^n\theta^i\otimes\nabla_{u_i}$ and $\nabla^{\prime\prime}=\sum_{i=1}^n\overline{\theta^i}
\otimes\nabla_{\overline{u_i}}$.

With these notions understood, we collect some well-known facts in the following
\begin{lemma}\label{wellknownlemma}
Assume that $X$ and $\xi$ are respectively complex-valued vector field and $1$-form on $(M,g,J)$.
\begin{enumerate}
\item
A (real) killing vector field is real holomorphic. If $X$ is real holomorphic, then so is $JX$. If $X\longleftrightarrow\xi,$ then $JX\longleftrightarrow J\xi.$ Here the action of $J$ on $1$-forms $\xi$ is canonically defined by $J(\xi)(Y):=-\xi(JY)$ for any vector field $Y$.

\item
If $X$ is of type $(1,0)$ and $X\longleftrightarrow\xi$ \rm{(}$\xi$ is necessarily a $(0,1)$-form\rm{)}, then $X$ is holomorphic if and only if $\nabla^{\prime\prime}\xi=0$.

\item
If $X$ is real holomorphic and $X\longleftrightarrow\xi$, then $\xi$ can be decomposed in a unique manner as $$\xi=\xi^H+dh_1+d^ch_2=\xi^H+dh_1+J(dh_2),$$
where $\xi^H$ is the harmonic part of $\xi$, $d^c:=\sqrt{-1}(\bar{\partial}-\partial)$, and $h_i$ $(i=1,2)$ are real-valued functions with vanishing integral. Moreover, $X$ is killing if and only if $h_1=0$.

\item
If $\xi=\sum_{i=1}^n\alpha_i\overline{\theta^i}$ is a $(0,1)$-form, then
\be\label{laplacian}
\Delta_1(\xi)=2\big[(\nabla^{\prime\prime})^{\ast}
\nabla^{\prime\prime}\xi+\sum_{i=1}^n\text{Ric}
(e_i,e_i)\alpha_i\overline{\theta^i}\big],\ee
where $\Delta_1$ is the Laplacian acting on $1$-forms in the notation of (\ref{delta}), $(\nabla^{\prime\prime})^{\ast}$ the formal adjoint of $\nabla^{\prime\prime}$ relative to the metric $g$, and $\text{Ric}(\cdot,\cdot)$ the Ricci tensor of $g$.
\end{enumerate}
\end{lemma}
\begin{proof}~
\begin{enumerate}
\item
The first part is quite well-known (cf. \cite[p. 107]{Mo} or \cite[Thm. 4.3]{Ko}). For the second part, only note that a real vector field $X$ is real holomorphic if and only if $\nabla_{JX}(Y)=J\nabla_X(Y)$ for all real vector fields $Y$ (cf. \cite[p. 6]{Ma}). For the third part,
$X\longleftrightarrow\xi$ is equivalent to $\xi(Y)=g(X,Y)$ for any $Y$. Thus $$(J\xi)(Y)=-\xi(JY)=-g(X,JY)=g(JX,Y).$$
 The last equality is due to the $J$-invariance of the K\"{a}hler metric $g$.
\item
See \cite[Prop. 4.1]{Ko}.

\item
See \cite[p. 131]{Mo} or \cite[Thm. 4.4]{Ko}. Note that $d^ch_2=J(dh_2)$ is due to the fact that the $(1,0)$-forms and $(0,1)$-forms are eigensubspaces of $J$ relative to the eigenvalues $-\sqrt{-1}$ and $\sqrt{-1}$ respectively.

\item
To the author's best knowledge, the Bochner-type formula (\ref{laplacian}) should be due to Lichnerowicz in \cite[\S 9]{Lic} (cf. \cite[p. 158]{Ko}). We refer the reader to  \cite[p. 310]{Wu} for a thorough treatment on this kind of formulas.
\end{enumerate}
\end{proof}

\subsection{Proof of Theorem \ref{mul}}
With Lemma \ref{wellknownlemma} in hand, we can now proceed to show Theorem \ref{mul}.
It is well-known that the first nonzero eigenvalue $\lambda_{2,0}(\mathbb{C}P^n(c),g_0)=(n+1)c$ whose multiplicity is exactly $n^2+2n$. Therefore we know through the assumptions made in Theorem \ref{mul} that \be\label{ric2}\text{Ric}(g)=\frac{s_g}{2n}g
=\frac{(n+1)c}{2}g\ee
and
\be\label{ric3}\text{$\lambda_{2,0}(g)=(n+1)c$ with multiplicity $n(n+2)$}.\ee

Let $f$ be an eigenfunction with respect to the first nonzero eigenvalue $\lambda_{2,0}(g)=(n+1)c$, i.e., $\Delta_0f=(n+1)cf$. First we have the following claim.

\textbf{Claim}. \emph{The real vector field dual to the $1$-form $J(df)$ is nontrivial and killing.}
\begin{proof}
$$\Delta_1(df)=(dd^{\ast}+d^{\ast}d)(df)=(dd^{\ast}d)f=d(dd^{\ast}+d^{\ast}d)f=
d\Delta_0f=(n+1)c(df),$$
which implies that \be\label{key4}\Delta_1(\bar{\partial}f)=(n+1)c(\bar{\partial}f)\ee as the Laplacian preserves the types of forms on compact K\"{a}hler manifolds.

Denote by $|\varphi|^2:=\int_Mg(\varphi,\overline{\varphi})$ the global squared norm of a form $\varphi$ on $(M,g,J)$. Now applying the Bochner-type formula (\ref{laplacian}) to the $(0,1)$-form $\bar{\partial}f$ and the facts (\ref{ric2}) and (\ref{key4}) yields
\be\label{key5}\lambda_{2,0}(g)|\bar{\partial}f|^2=
\int_Mg\big(\Delta_1(\bar{\partial}f),\overline{\bar{\partial}f}\big)
=2|\nabla^{\prime\prime}(\bar{\partial}f)|^2+
\frac{s_g}{n}|\bar{\partial}f|^2
\geq\frac{s_g}{n}|\bar{\partial}f|^2=(n+1)c|\bar{\partial}f|^2.
\ee

Note that $f$ is a real-valued non-constant function on $M$ and so $\bar{\partial}f$ is not identically zero, i.e., $|\bar{\partial}f|^2>0$. Coupling this with (\ref{key5}) imply that $\lambda_{2,0}(g)\geq(n+1)c$, which, together with (\ref{ric3}), tells us that the inequality (\ref{key5}) is indeed an equality. Therefore $\nabla^{\prime\prime}(\bar{\partial}f)=0$ and thus Lemma \ref{wellknownlemma} says that the $(1,0)$-type complex vector field dual to $\bar{\partial}f$ is nontrivial and holomorphic. Hence the real vector field dual to $df$, say $W$, is nontrivial and real holomorphic. Therefore $JW$ is also nontrivial and real holomorphic and $JW\longleftrightarrow J(df)$ due to Lemma \ref{wellknownlemma}. Note that the integral of $f$ vanishes as $\Delta_0(f)=(n+1)cf$. Thus still by Lemma \ref{wellknownlemma} we deduce that the real vector field dual to the $1$-form $J(df)$ is nontrivial and killing, which completes the proof of this claim.
\end{proof}

Since the multiplicity of $\lambda_{2,0}(g)=(n+1)c$ is $n^2+2n$ and so we have $n^2+2n$ linearly independent eigenfunctions $f_i$ $(1\leq i\leq n^2+2n)$ and hence $n(n+2)$ linearly independent killing vector fields $JW_i$, where $W_i\longleftrightarrow df_i$ $(1\leq i\leq n^2+2n)$.
In summary, we conclude that under the conditions assumed in Theorem \ref{mul}, the dimension of the isometric group of the compact K\"{a}hler manifold $(M,g,J)$ is no less than $n^2+2n$.

Recall an old result of Tanno (\cite{Ta}) that the dimension of the automorphism group of an almost Hermitian manifold preserving both the Hermitian metric and the almost-complex structure is no larger than $n^2+2n$, with equality if and only if it is a standard complex projective space. Note in Lemma \ref{wellknownlemma} that those $n^2+2n$ linearly independent killing vector fields $JW_i$ on the compact K\"{a}hler manifold $(M,g,J)$ are automatically real holomorphic, i.e., preserve the complex structure $J$. Thus we yield the desired conclusion via Tanno's above-mentioned result.


\begin{thebibliography}{99999}
\bibitem[Ap55]{Ap}
{M. Apte}:
\newblock {\em Sur certaines classes caract$\acute{e}$ristiques des vari$\acute{e}$t
$\acute{e}$s K\"{a}hl$\acute{e}$riennes compactes},
\newblock  C. R. Acad. Sci. Paris {\bf 240} (1955), 149-151.

\bibitem[Be68]{Be}
{M. Berger}:
\newblock {\em Le spectre des vari\'{e}t\'{e}s riemanniennes},
\newblock Rev. Roumaine Math. Pures Appl. {\bf 13} (1968), 915-931.

\bibitem[CV80]{CV}
{B.-Y. Chen, L. Vanhecke}:
\newblock {\em The spectrum of the Laplacian of K\"{a}hler manifolds},
\newblock  Proc. Amer. Math. Soc. {\bf 79} (1980), 82-86.

\bibitem[Fu18]{Fu}
{K. Fujita}:
\newblock {\em Optimal bounds for the volumes of K\"{a}hler-Einstein Fano manifolds},
Amer. J. Math. {\bf 140} (2018), 391-414.

\bibitem[GG86]{GG}
{H. Gauchman, S.I. Goldberg}:
\newblock {\em  Spectral rigidity of compact Kaehler and contact manifolds},
\newblock Tohoku Math. J. (2) {\bf 38} (1986), 563-573.

\bibitem[Go84]{Go}
{S.I. Goldberg}:
\newblock {\em A characterization of complex projective spaces},
\newblock C. R. Math. Rep. Acad. Sci. Canada {\bf 6} (1984), 193-198.

\bibitem[Ko72]{Ko}
{S. Kobayashi}:
\newblock {\em Transformation Groups in Differential Geometry},
\newblock Classics in Mathematics, Springer-Verlag, Berlin, 1995, Reprint of the 1972 edition.

\bibitem[Li18]{Li3}
{P. Li}:
\newblock{\em The spectral rigidity of complex projective spaces, revisited},
\newblock Math. Z., doi:10.1007/s00209-018-2055-8.

\bibitem[Lic69]{Lic}
{A. Lichn\'{e}rowicz}:
\newblock{\em Vari\'{e}t\'{e}s k\"{a}hl\'{e}rienees et premi\`{e}re classe de Chern}.
\newblock J. Differential Geom. {\bf 1} (1967), 195-223.

\bibitem[Ma71]{Ma}
{Y. Matsushima}:
\newblock{\em Holomorphic Vector Fields on Compact K\"{a}hler Manifolds}.
\newblock in: Conference Board of Mathematical Sciences Regional Conference Series in Mathematics, vol. 7, Amer. Math. Soc., Providence, RI, 1971, vi+38 pp.

\bibitem[MS67]{MS}
{H.F. McKean, I.M. Singer}:
\newblock {\em Curvature and the eigenvalues of the Laplacian},
\newblock J. Differential Geom. {\bf 1} (1967), 43-69.


\bibitem[Mi64]{Mi}
{J. Milnor}:
\newblock {\em  Eigenvalues of the Laplace operator on certain manifolds},
\newblock  Proc. Nat. Sci. U.S.A. {\bf 51} (1964), 542.

\bibitem[Mo07]{Mo}
{A. Moroianu}:
\newblock {\em Lectures on K\"{a}hler Geometry},
\newblock London Mathematical Society student Texts 69, Cambridge University Press, Cambridge, 2007.

\bibitem[Pa70]{Pa}
{V.K. Patodi}:
\newblock {\em Curvature and the fundamental solution of the heat operator},
\newblock J. Indian Math. Soc. {\bf 34} (1970), 269-285.

\bibitem[Pa96]{Pa2}
{V.K. Patodi}:
\newblock {\em Collected papers of V.K. Patodi},
\newblock World Scientific Publishing Co., Inc., River Edge, NJ, 1996.

\bibitem[Sa71]{Sa}
{T. Sakai}:
\newblock {\em On eigenvalues of Laplacian and curvature of Riemannian manifold},
\newblock Tokoku. Math. J. {\bf 23} (1971), 589-603.

\bibitem[Ta69]{Ta}
{S. Tanno}:
\newblock {\em The automorphism groups of almost Hermitian manifolds},
\newblock Trans. Amer. Math. Soc. {\bf 137} (1969), 269-275.

\bibitem[Ta73]{Ta1}
{S. Tanno}:
\newblock {\em Eigenvalues of the Laplacian of Riemannian manifolds},
\newblock Tokoku. Math. J. {\bf 25} (1973), 391-403.

\bibitem[Ta74]{Ta2}
{S. Tanno}:
\newblock {\em The spectrum of the Laplacian for $1$-forms},
\newblock Proc. Amer. Math. Soc. {\bf 45} (1974), 125-129.

\bibitem[Ti00]{Ti}
{G. Tian}:
\newblock{\em Canonical metrics in K\"{a}hler geometry},
\newblock Notes taken by Meike Akveld, Lectures in Mathematics ETH Z\"{u}rich. Birkh\"{a}user Verlag, Basel,
2000.

\bibitem[Wu88]{Wu}
{H.-H. Wu}:
\newblock {\em The Bochner technique in differential geometry},
\newblock  Math. Rep. 3 (1988), no. 2, i-xii and 289-538.
\end{thebibliography}
\end{document}